\documentclass[a4paper]{amsart}

\usepackage[english]{babel}

\usepackage{amssymb}

\usepackage[latin1]{inputenc}
\usepackage[T1]{fontenc}

\newcommand{\Z}{\mathbb{Z}}
\newcommand{\Q}{\mathbb{Q}}

\newtheorem{theo}{Theorem}

\theoremstyle{definition}

\theoremstyle{remark}
\newtheorem*{rem}{Remark}

\begin{document}
\title{Elliptic curves and biquadrates}

\author[J. Aguirre]{Juli\'an Aguirre}
\address{Departamento de Matem\'aticas\\
Universidad del Pa\'is Vasco, UPV/EHU\\
Aptdo. 644, 48080 Bilbao, Spain}
\email[J. Aguirre]{julian.aguirre@ehu.es}

\author[J.C. Peral]{Juan Carlos Peral}
\address{Departamento de Matem\'aticas\\
Universidad del Pa\'is Vasco, UPV/EHU\\
Aptdo. 644, 48080 Bilbao, Spain}
\email[J.C. Peral]{juancarlos.peral@ehu.es}

\thanks{J. Aguirre supported by grant IT-305-07 of the Basque Government. J.C. Peral  supported by the UPV/EHU grant EHU 10/05.}


\begin{abstract}
Given two integers $m$ and $n$ consider $N= m^4+ n^4$ and  the elliptic curve 
\[
y^2= x^3-N\,x
\]
The rank  of this family over $\Q(m,n)$ is at least $2$.

Euler constructed a parametric family of  integers $N$ expressible in two different ways  as a sum of two biquadrates. We prove that for those $N$  the corresponding family of elliptic curves has rank at least $4$ over $\Q(u)$. This is an improvement on previous results of Izadi, Khoshnam and Nabardi.
\end{abstract}

\maketitle

\section{Sums of two biquadrates and  elliptic curves}

\subsection{Elliptic curves  and sums of biquadrates. General case}\label{general}
For integers $m$ and $n$ consider   the elliptic curves given by 
\begin{equation}\label{equ}
y^2=x^3- (m^4+n^4)x .
\end{equation}

The torsion groups for elliptic curves  of the form $y^2=x^3+ D\,x$ are, $\Z/ 4 \Z$ for $D=4$, $\Z/ 2\Z\times \Z/ 2\Z  $ for $D=-h^2$ and $\Z/2\Z$ in the rest of the cases (see \cite{S}, Proposition $6.1$, p.~311).  This together with the  fact that $h^2= m^4+n^4$ has no solution (see \cite{HW}, Theorem~266)  implies that the torsion group for  the  curves \eqref{equ} is always $\Z/2\Z$. We have proved the following

\begin{theo}
The family \eqref{equ} has rank at least $2$ over $\Q(m,n)$.
\end{theo}
 
To estimate the rank we study the existence of  generic points on these curves and their associates. The identity 
\[
(-1)(n)^4+\frac{-(m^4+n^4)}{-1}\,(1)^4=m^4
\]
reveals that for $d=-1$ the homogenous space $(U)^4 d+(-N/d)V^4=H^2$ has the solution $(U,V,H)=(n,1,m^2)$, and this implies that the point $P_1=(- n^2, m^2 n)$ is on the curve.

For the associated curve $y^2=x^3+ 4 (m^4+ n^4) x$ we have  the identity
\[
(2)(m+n)^4+ \frac{4(m^4+ n^4)}{2}=(2\,(m^2+ n^2+ m n))^2,
\]
so that for $d=2$  the homogenous space $(U)^4 d+ (4\,N/d) V^4=H^2$ has the solution $(U,V,H)=(m+n,1,2 (m^2+ n^2+ m\,n))$.

Consequently the point $Q_1=(2 (m + n)^2,4 (m + n) (m^2 + m\,n + n^2))$ is on the curve. We transfer $Q_1$ to the main curve,  and we have two points on that curve given by:
\begin{align*}
P_1(m,n) &= \{- n^2, m^2 n\},\\
P_2(m,n) &=\Bigl\{ \frac{(m^2 + m\,n + n^2)^2}{(m + n)^2}, \frac{m\,n(m^2 + m\,n + n^2) (2 m^2 + 3\,m\,n + 2\,n^2)}{(m + n)^3}\Bigr\}.
\end{align*}

In order to prove that the family \eqref{equ} has rank at least $2$ over $\Q(m,n)$, 
it suffices to find a specialization $(m,n)=(m_0, n_0)$ such that the points  $P_1(m_0, n_0),$ and  $P_2(m_0, n_0)$ are independent points on the specialized curve over 
$\Q$, due to the fact that the specialization is a homomorphism (see \cite{S}.)
For $(m_0,n_0)=(2,1)$ the curve is $y^2=x^3-17\,x$ and the points are 
\begin{align*}
P_1(2,1) &= \{-1,4\},\\
P_2(2,1) &=\Bigl\{ \frac{49}{9}, \frac{224}{27}\Bigr\}.
\end{align*}
The curve has rank $2$ over $\Q$ and the points are independent as can be checked with  the program \texttt{mwrank} \cite{mwrank}.
So  the rank of the family \eqref{equ} is at least $2$ over $\Q(m,n)$.

\subsection{Examples with bigger rank}
We have found, after a short search, many examples of curves with rank $7$ over $\Q$,  eight examples of curves with rank $8$ over $\Q$ and one example with rank $9$. The smaller values of $N$ that we found, in the case of rank $7$, are:
  \begin{align*}
3534242722&= 83^4+243^4\\
 3730925026&=125^4+243^4\\
 3732157186&= 155^4+237^4\\
3840351442&=147^4+241^4\\
  9633078002&= 77^4+313^4\\
 26939353666&= 77^4+405^4\\
 71486456242&= 81^4+517^4
    \end{align*} 
    The curves with rank  $8$  correspond to the following values of $N$:
    \begin{align*}
25792915457&= 326^4+347^4\\
141262310897&=88^4+613^4\\
436341291697&=631^4+726^4\\
9788096042497&=972^4+1727^4\\
106232596858561&=491^4+3210^4\\
159764080671457&=1191^4+3544^4\\
202891791817457&=1652^4+3739^4\\
380344532478577&=3513^4+3886^4
 \end{align*}
 Finally,  for $N=228746044559762=2387^4+3743^4$ the curve $y^2=x^3-N\,x$ has rank $9$.

\subsection{Euler parametrization} \label{euler}
The problem of  finding integers expressible in two different ways as sum of two fourth powers has been studied by several authors, see  the second volume  of the  History of Number Theory  (\cite{DI}, p. 644--648).
In particular, Euler constructed a bi-parametric family of solutions of $N=A^4+ B^4 =C^4+D^4$  given as follows:
\begin{align*}
A(u,w)&=u (u^6 + u^4 w^2 - 2\,u^2 w^4 + 3\,u\,w^5 + w^6),\\
B(u,w)&=w (u^6 - 3\,u^5 w - 2\,u^4 w^2 + u^2 w^4 + w^6),\\
C(u,w)&=u (u^6 + u^4 w^2 - 2\,u^2 w^4 - 3\,u\,w^5 + w^6),\\
D(u,w)&=w (u^6 + 3\,u^5 w - 2\,u^4 w^2 + u^2 w^4 + w^6).
\end{align*}
In what follows we  take $w=1$. There is no loss of generality with this choice due to the homogeneity. Then
\begin{equation}\label{eq_N}
\begin{aligned}
N&= A(u,1)^4+B(u,1)^4\\
&=C(u,1)^4+D(u,1)^4\\
&=(1+6\,u^2+u^4)(1-u^4 + u^8)(1-4\,u^2+8\,u^4-4\,u^6+u^8)\\
&\qquad\times (1+2 u^2+11\,u^4+2\,u^6+u^8).
\end{aligned}
\end{equation}

\subsection{Elliptic curves  and sums of biquadrates. Case of two equal sums}
Now we consider the  curves 
\begin{equation}\label{equ2}
y^2=x^3- N\,x
\end{equation}
where $N$ is given by~\eqref{eq_N}. Subsection~\ref{general} and  the fact that $N$ can be expressed in two different ways as sum of two biquadrates suggest that the  subfamily \eqref{equ2} has  a greater rank than the general family \eqref{equ}, and that this  generic rank has to be at least $4$. 
 
In \cite{IKN} the authors prove that if $N$ is square free, then the rank  of \eqref{equ2} over $\Q(u)$ is at least $3$. They also prove, assuming the Parity Conjecture and some other conditions, that the rank is at least $4$. We improve their result by proving in the next theorem that the rank is at least $4$ over $\Q(u)$ unconditionally.
 
\begin{theo}
The rank  of \eqref{equ2} over $\Q(u)$ is at least $4$.    
\end{theo}

\begin{proof}
 We proceed, as in the general case, searching  solutions in the homogenous spaces $d \,U^4-( N/d )V^4= H^2$ of the curve, where $d$ is a divisor of $N$.
 
There is one solution for $(U,V)=(1,1)$ and the divisor $d=(1-u^4+u^8) (1+2\,u^2+11\,u^4+2\,u^6+u^8)$, and  another one for 
 $(U,V)=(A(u),1)=(u(1+3\,u-2\,u^2+u^4+u^6),1)$ for the divisor $d=-1$. 
These two solutions  produce the corresponding  two points on the curve given by
\begin{align*}
P_1(u)&=\bigl\{(1-u^4+u^8)(1+2\,u^2+11\,u^4+2\,u^6+u^8),\\
&\qquad u^2 (-5+4\,u^2+u^4+u^6)(1-u^4+u^8)(1+2\,u^2+11\,u^4+2\,u^6+u^8)\bigr\},\\
P_2(u)&=\bigl\{-u^2 (1+3\,u-2\,u^2+u^4+u^6)^2,\\
&\qquad u(1+3\,u-2\,u^2+u^4+u^6)(1+u^2-2\,u^4-3\,u^5+u^6)^2\bigr\}.
\end{align*}
  
We perform  a similar search in the homogenous spaces of the associated curve $y^2=x^3+ 4\,N\,x$. In this case we observe the existence of solutions for $(U, V)=(A(u)+B(u),1)=(1+u+4\,u^2-2\,u^3-2\,u^4-2\,u^5+u^6+u^7,1)$ and $d=2$, and for $(U, V)=(u,1)$ and $d=4(1+6\,u^2+u^4)(1+2\,u^2+11\,u^4+2\,u^6+u^8)$.
  
Again we have the corresponding points on the associated curve given by
\begin{align*}
Q_1(u)&=\bigl\{2(1+u+4\,u^2-2\,u^3-2\,u^4-2\,u^5+u^6+u^7)^2,\\
&\qquad4(1+u+4\,u^2-2\,u^3-2\,u^4-2\,u^5+u^6+u^7)\\
&\qquad\times(1+u+6\,u^2+5\,u^3+5\,u^4-21\,u^5-5\,u^6-2\,u^7+13\,u^8 \\
&\qquad\qquad+15\,u^9-u^{10}-7\,u^{11}+u^{13}+u^{14})\bigr\}\\
Q_2(u)&=\bigl\{4\,u^2(1-u^4+u^8)(1-4\,u^2+8\,u^4-4\,u^6+u^8),\\
&\qquad\times 4\,u(1-u^4+u^8)(1-4\,u^2+8\,u^4-4\,u^6+u^8)\\
&\qquad\times(1+4\,u^2+6\,u^4+3\,u^6-4\,u^8+2\,u^{10})\bigr\}.
\end{align*}
We transfer  points $Q_1(u)$ and $Q_2(u)$ to the original curve and jointly with $P_1$ and $P_2$, we get four points  whose  $x$-coordinates are:
\begin{align*}
x_1(u)&=(1-u^4+u^8)(1+2\,u^2+11\,u^4+2\,u^6+u^8),\\
x_2(u)&=-u^2 (1+3\,u-2\,u^2+u^4+u^6)^2,\\
x_3(u)&=\frac{(1+4\,u^2+6\,u^4+3\,u^6-4\,u^8+2\,u^{10})^2}{4\,u^2},\\
x_4(u)&=(1+u+6\,u^2+5\,u^3+5\,u^4-21\,u^5-5\,u^6-2\,u^7+13\,u^8\\
&\qquad +15\,u^9-u^{10}-7\,u^{11}+u^{13}+u^{14})^2\\
&\qquad\qquad/(1+u+4\,u^2-2\,u^3-2\,u^4-2\,u^5+u^6+u^7)^2.
\end{align*}
  
We finish proving that the rank is at least $4$, arguing as before.  We choose $u=2$. The  curve is $y^2= x^3-635318657\,x$ and its rank is $4$. The points are
\begin{align*}
P_1&=\{137129, 49914956  \},\\
P_2&=\{-24964, 549998 \},\\
P_3&=\Bigl\{\frac{1766241}{16},\frac{ 2285325807}{64}\Bigr\},\\
P_4&=\Bigl\{\frac{365689129}{9801},\frac{ 5156125463944}{970299}\Bigr\}.
 \end{align*}
A calculation with \texttt{mwrank} \cite{mwrank} shows that the four points are independent. This implies, by an specialization argument, that the rank of the family over $\Q(u)$ is at least $4$, see \cite{S}.

An alternative proof follows  by considering directly the steps in the  $2$-descent argument for curves of the shape $y^2=x^3+A\,x^2+B\,x$, see \cite{ST}.
\end{proof}
 
\begin{rem}
Imposing conditions that force the root number to be equal to $-1$ it is possible to find a subfamily  with rank at least $5$ over $\Q(u)$ but in this  case the result is conditional to the validity of  the Parity Conjecture.
\end{rem}

\begin{rem}
The construction given above suggest that if we had families of integers $N$  with many  different representations as sum of two fourth powers, then it would be possible to  construct families of elliptic curves, and examples of elliptic curves, having large ranks. Unfortunately, not even a single example with three different representations is known, see \cite{G}. Moreover, a simple  heuristic  density argument  tell us that the possibility to find multiple representations of that  form it is very unlikely.
\end{rem}
 
\subsection{Examples with bigger rank}
In \cite{IKN} the authors quote several examples of rank $8$ curves within this family, but all  of thems reduce to one because the corresponding values of $N$ differ by a fourth power factor. It is  the first one in the following list. We have found another  value of $N$, the second one, for which the corresponding curve also has rank $8$ over $\Q$. 
\begin{alignat*}{2}
    155974778565937 &= 1623^4+3494^4  & &= 2338^4+3351^4 \\
2701104520630058561 &= 2513^4+40540^4 & &= 11888^4+40465^4
 \end{alignat*}  
 
Computations using this family are highly time-consuming because of the size of the coefficients.

\begin{rem}  Curves of the form $y^2=x^3+ B\,x$ have $j$-invariant equal to $1728$ and have been studied  by many authors, see for example \cite{SH}, \cite{F}, \cite{N} and  \cite{ACP} and the references given there. An example with rank $14$ was found by Watkins within the family  constructed in  \cite{ACP}.
\end{rem}

\begin{rem}
A more detailed version of this note will appear elsewhere.
\end{rem}

 \end{document}